\documentclass[journal]{IEEEtran}
		\usepackage[utf8]{inputenc}
		\usepackage{cite}
		\usepackage{graphicx}
		\usepackage{tikz}
		\usepackage{pgfplots}
		\pgfplotsset{compat=newest,legend style={font=\footnotesize},
                     ticklabel style={font=\footnotesize},
                     x label style={font=\footnotesize},
                     y label style={font=\footnotesize}}
		\usepackage{amsmath,amssymb,amsthm,mathtools}
		\usepackage{url}
		\usepackage{array}
		\usepackage{hyperref}

    \usepackage{enumerate}
    \DeclareUnicodeCharacter{2212}{−}

\def\X{\mathcal{X}}
\def\Xdual{\mathcal{X}'}
\def\H{\mathcal{H}}
\def\measop{\boldsymbol{\nu}}
\def\reals{\mathbb{R}}
\def\nnreals{\reals_+}
\def\meas{\mathbf{y}}
\def\dualitymap{\mathcal{J}_\X}
\def\dualitymapdual{\mathcal{J}_{\Xdual}}
\def\Rieszmap{\mathcal{R}}
\def\Ell{L}
\def\f{f}
\def\matrix{\mathbf{H}}
\def\matrixA{\mathbf{A}}
\def\transform{\mathbf{L}}

\def\dcoeffs{\mathbf{f}}

\def\rcoeffs{\mathbf{a}}

\def\solutionmap{S}

\def\sign{\operatorname{sign}}
\def\dx{\mathrm d}
\def\LipConstant{K}
\def\el{e}
\newcommand{\LinOp}[2]{\mathcal{L}\left(#1,\,#2\right)}
\def\T{\mathrm{T}}
\def\Id{\mathbf{Id}}

\newtheorem{theorem}{Theorem}
\newtheorem{lemma}[theorem]{Lemma}
\newtheorem{definition}[theorem]{Definition}

\newtheorem{proposition}[theorem]{Proposition}

\newtheorem{example}[theorem]{Example}
\newtheorem{remark}[theorem]{Remark}

\title{Stability of Image-Reconstruction Algorithms}
\author{Pol~{del Aguila Pla},~\IEEEmembership{Member,~IEEE,}\thanks{
        Pol~{del Aguila Pla} is with the CIBM Center for Biomedical Imaging, in
        Switzerland.}
        Sebastian~Neumayer, and
        Michael~Unser,~\IEEEmembership{Fellow,~IEEE}
        \thanks{The authors are with the Biomedical Imaging Group
        at the École polytechnique fédérale de Lausanne, in Lausanne,
        Switzerland.}}

\begin{document}
\bstctlcite{IEEEexample:BSTcontrol}

\maketitle

\begin{abstract}
    % !TEX root = ../main.tex
Robustness and stability of image-reconstruction
algorithms have recently come under scrutiny. Their
importance to medical imaging cannot be overstated.
We review the known results for the topical variational
regularization strategies ($\ell_2$ and $\ell_1$
regularization) and present novel stability results for
$\ell_p$-regularized linear inverse problems for
$p\in(1,\infty)$. Our results guarantee Lipschitz continuity
for small $p$ and H\"{o}lder continuity for larger $p$. They
generalize well to the $\Ell_p(\Omega)$ function spaces.

\end{abstract}

\section{Introduction} \label{sec:intro}
    % !TEX root = ../main.tex
\IEEEPARstart{I}{nverse} problems are at the core of computational imaging.
Medical imaging critically depends on the guarantees provided by established
image-reconstruction methodologies to inform diagnostic and treatment
decisions. New techniques based on artificial intelligence and deep neural
networks offer major average performance improvements in most
applications~\cite{JinMcCFrouUns2017,ZhuLiuCauRosRos2018,Str2018,AguilaPla2019}, at the cost
of poor practical stability~\cite{Antun2020,HN20} and a lack in theoretical
guarantees. In particular, seemingly small perturbations of the measurements
can produce large errors in the resulting image. Insidiously, these errors may
incorporate deceptive patterns that look realistic because they were learnt
from the training database (hallucination)~\cite{GotAntAdcHan2020}.
Additionally, questions have also been raised on the stability guarantees
provided by variational inverse-problem approaches using
$\ell_p$-regularization strategies to induce
structure~\cite{Darestani2021,Genzel2020}. In this paper, we first review the
unicity and stability properties of classical Tikhonov regularization ($p=2$)
and sparsity-promoting regularization ($p=1$). Then, we present novel
stability results for $\ell_p$-regularized inverse problems for
$p\in(1,\infty)$. In particular, we show that the solution map is locally
Lipschitz continuous for $p\in(1,2]$ and globally $1/(p-1)$-H\"{o}lder
continuous for $p\in (2,\infty)$. The proofs also cover the case of the
$\Ell_p(\Omega)$ function spaces. Our aim in presenting these
results is to pave the way toward a quantitative comparison of
image-reconstruction methods in terms of stability.

A broad category of image-reconstruction algorithms can be formulated as the
variational problem
\cite{Engl1996,Vogel2002,Scherzer09,Schuster2012,Benning2018,McCann2019}
\begin{equation} \label{eq:image_rec}
    \min_{\tilde\dcoeffs\in\reals^N} \Bigl\lbrace
            \bigl\Vert \meas - \tilde\matrixA\tilde\dcoeffs \bigr\Vert_2^2
            + \lambda \bigl\Vert \transform \tilde\dcoeffs \bigr\Vert_{p}^{p}
                        \Bigr\rbrace\,.
\end{equation}
For $p=2$, \eqref{eq:image_rec} corresponds to classical Tikhonov
regularization~\cite{Tik43,Tikhonov1977}; and for $p=1$, \eqref{eq:image_rec}
corresponds to sparsity-based regularization~\cite{Candes2006,Donoho2006}.
Here, $\tilde\matrixA\in\reals^{M\times N}$  with $M\leq N$ is the forward
operator. For a given imaging system, it relates the discrete image
representation $\tilde\dcoeffs\in\reals^N$ to the measurements
$\meas\in\reals^M$. Furthermore, $\transform$ is a linear transform
(\emph{e.g.}, the finite-difference operator) that gets penalized through the
$\ell_p$ norm, and $\lambda\in\nnreals$ is the regularization parameter
controlling the tradeoff between the data-fidelity term and the regularizer.
An alternative formulation to \eqref{eq:image_rec} is the synthesis formulation
\begin{equation} \label{eq:synthesis_form}
    \min_{{\dcoeffs}\in\reals^N} \bigl\lbrace
            \Vert \meas - {\matrixA} {\dcoeffs} \Vert_2^2
            + \lambda \left\Vert {\dcoeffs} \right\Vert_{p}^{p}
                        \bigr\rbrace\,,
\end{equation}
which, if $\transform$ is invertible, corresponds exactly to
\eqref{eq:image_rec} with ${\dcoeffs} = \transform \tilde\dcoeffs$ and
${\matrixA}=\tilde\matrixA\transform^{-1}$. Beyond image reconstruction
(\emph{e.g.}, for template-based reconstruction
methods~\cite{CO18,NPS18,LNOS19,NT2021}), this type of variational problem
appears in, for example, statistics under the name of bridge
regression~\cite{Fu1998,WanWenMal20} for $p\in(1,2)$,
and in machine learning as part of the multiple-kernel
learning~\cite{KloGil2012} literature.

The objective of this paper is to study the robustness of the reconstruction of
$\dcoeffs$ from $\meas$ based on \eqref{eq:synthesis_form}. Although concrete
definitions of robustness, stability, and similar concepts vary, the
predominant view in the literature is that robustness should be measured in
terms of the continuity properties of the reconstruction (or solution) map
$\solutionmap\colon\reals^M\rightarrow\reals^N$. This map is only well defined
if \eqref{eq:synthesis_form} has a unique solution $\dcoeffs_\meas$, in which
case $\solutionmap\colon \meas \mapsto \dcoeffs_\meas$. In this context, we
study the stability of the reconstruction in terms of bounds on
$\Vert \dcoeffs_{\meas_1} - \dcoeffs_{\meas_2} \Vert_{\ell_p}$ with respect to
$\Vert \meas_1 - \meas_2 \Vert_{2}$ for any two measurements $\meas_1$ and
$\meas_2$. Depending on the relation between these terms, the stability is
weaker or stronger. The most general category we contemplate for stability is
local H\"{o}lder continuity, where
\begin{equation} \label{eq:local_Holder}
   \Vert \dcoeffs_{\meas_1} - \dcoeffs_{\meas_2} \Vert_{\ell_p} \leq
    K \, \Vert \meas_1 - \meas_2 \Vert^\beta_{2}\,,
\end{equation}
with $K\geq0$ and $\beta\in[0,1]$ for any two measurements 
$\meas_1,\meas_2\in \reals^M$ within a set of measurements $Y\subset\reals^M$.
Here, local signifies that $K=K(Y)$ depends on the choice of set $Y$, which may be 
a cube or ball in $\reals^M$ containing expected reasonable measurements. 
The strongest stability result comprised within the same expression 
\eqref{eq:local_Holder} is global Lipschitz continuity, where $Y=\reals^M$ and 
$\beta=1$. Given
bounds such as \eqref{eq:local_Holder} for any two image-reconstruction algorithms,
one can objectively compare their stability properties in terms of the exponent
$\beta$ and the value of $K$.

\subsection{Related Work}

    Although the robustness of regularized variational problems has been
    studied extensively before, most studies relied on asymptotic criteria
    for vanishing noise~\cite{BurOsh04,DauDefDem04,ResSch06,GraHalSch08}.
    These are valid only when $\Vert\meas_1-\meas_2\Vert_2\rightarrow0$ and
    are weaker than the ones we target under the conditions stated
    in~\eqref{eq:local_Holder}. (See~\cite{Bonnans2000} for an extensive
    overview on stability criteria for variational problems.)

    The stability of solutions of variational inverse problems has also been
    investigated using criteria similar to ours. In~\cite{DurNik06}, the
    authors assume that the forward operator is injective and invertible and
    that directional derivatives do not vanish at non-smooth points of the
    objective functional. These conditions are rather restrictive and
    superfluous in our particular setting. In~\cite{Shv12}, the authors
    consider finite-dimensional constrained-optimization problems and use
    ideas similar to ours. However, our analysis builds on a condition that
    involves the modulus of convexity of the regularizer and that is less
    limiting than the strong convexity imposed in~\cite{Shv12}.

    A related but fundamentally different problem than the one we discuss
    is algorithmic stability in learning theory. There, the interest is to
    bound the magnitude of changes in the output of a learned algorithm with
    respect to changes in its training data. In that context, $\ell_1$ and
    $\ell_p$ regularization have also been studied in
    detail~\cite{XuCarMan2012}.
    % Discarded: WibRosPog09, not peer reviewed.
    % Of particular relevance is~\cite{WibRosPog09}, which contains an
    % algorithmic stability result obtained using strategies similar to
    % ours.

\section{Variational Regularization of Inverse
         Problems} \label{sec:representer-theorem} \label{sec:assumptions}
        % !TEX root = ../main.tex
We now discuss the variational regularization of linear inverse
problems from the perspective used
in~\cite{Unser2016,Gupta2018,Unser2019,Unser2020a,Unser2021}.
The theory is formulated for Banach spaces, which are complete vector spaces
with a norm. This level of generality is appropriate for our study because
the $\ell_p$ spaces that characterize~\eqref{eq:synthesis_form} are
Banach spaces. Throughout the main body of the paper, we
rely on the intuitive understanding of some of the mathematical
terms, without diverting the reader's attention with extensive technical
details. Appendix~\ref{app:premath} is designed to complement the paper
by providing the basic functional analytical background for our work.

An image $\f$ is considered as an object in a Banach space $\X$.
The measurements
$\meas\in\reals^{M}$ of $\f$ are modeled as some noisy version of
$\measop(\f)$, where $\measop\colon\X\to \reals^{M}$ is a linear operator
given by
\begin{equation}\label{eq:measop}
    \f \mapsto \measop(\f) \coloneqq
     \bigl(\langle \nu_1, \f  \rangle_{\Xdual\times\X}, \ldots, \langle
      \nu_M, \f \rangle_{\Xdual\times\X}\bigr),
\end{equation}
for the set $\lbrace\nu_m\rbrace_{m=1}^M\subset\Xdual$ of linearly independent
measurement functionals. Here, the $\nu_m$ are elements of the dual space
$\Xdual$, which is made of all the linear continuous functionals
$\nu\colon\X\to\reals$ (see Definition~\ref{def:dual}). The notation
$\langle\nu_m,\f\rangle_{\Xdual\times\X}$ is used to denote the evaluation of
$\nu_m$ at $\f$. The operator $\measop$ in~\eqref{eq:measop} generalizes
the role of the matrix $\matrixA$ in~\eqref{eq:synthesis_form}.
The choice of the pair $(\X,\Xdual)$ for a specific problem corresponds
to the choice of regularizer and, thus, to the choice of the desired
properties of $\f$ (see \eqref{eq:solutionset} below). Although the theory is
more general~\cite{Unser2021} we make here the restrictive assumption that
$\X$ is a reflexive and strictly convex Banach space (see
Definition~\ref{def:reflexive}). This is true for the spaces of interest in
this paper, which are $\X=\ell_p$ and $\Xdual=\ell_q$ with $p \in(1,\infty)$
and $q=p/(p-1)$. Then, the solutions to the variational problem
\begin{equation}\label{eq:solutionset}
    \min_{\f\in\X}\left\lbrace
                        E\left( \meas, \measop(\f) \right)
                        + \psi\left( \Vert \f \Vert_{\X}\right)
                     \right\rbrace
\end{equation}
are taken to be reconstructions of $\f$ from the measurements $\meas$.

%\added{For example, in the case of two-dimensional computed tomography (CT)
%problems, one can choose the Sobolev space $\X=H_0^{1}(\Omega)$ for some
%compact $\Omega \subset \reals^2$. For a CT problem with $M$ measurements at
%angles $\lbrace\theta_m\rbrace_{m=1}^M\subset[0,\pi)$ and offsets
%$\lbrace t_m\rbrace_{m=1}^M\subset \reals$, we then have that the measurement
%functionals $\nu_m\in\Xdual$ for $m\in\lbrace1,2,\dots,M\rbrace$ can be
%identified with
%$\nu_m = \delta(t_m-\langle \cdot, \boldsymbol{\theta}_m\rangle)$,
%where $\boldsymbol{\theta}_m = [\cos(\theta_m),\sin(\theta_m)]$ (see
%Proposition~\ref{prop:Radon}). Then, the regularizer in \eqref{eq:solutionset}
%can be chosen as the classical Tikhonov regularizer
%$\Vert\nabla\f\Vert^2_{\Ell_2(\Omega)}$.
%In this paper, however, we will focus on the choices $\X=\ell_p$ or
%$\X=\Ell_p(\Omega)$, which correspond to $\ell_p$-regularization.}

The function $E\colon\reals^M \times \reals^M \to \overline{\reals}_+$ is a
data-fidelity term. It penalizes reconstructions $\f$ that do not agree with
the measurements $\meas$; for example, this could be the least-squares term
used in \eqref{eq:synthesis_form}. The function $E$ is assumed to be lower
semi-continuous, proper, and strictly convex in its second argument. The
function $\psi\colon\nnreals\to \nnreals$ is assumed to be strictly increasing
and strictly convex. It regulates how much we penalize $\f$ according to its
norm $\Vert \f \Vert_{\X}$. In general, more than one reconstruction $\f$ may
achieve the same minimum cost
\begin{equation} \label{eq:cost_functional}
        J(\meas,\f) = E\left( \meas, \measop(\f) \right) +
                              \psi(\Vert \f \Vert_{\X})\,.
\end{equation}
Such reconstructions are considered as equally good for the variational
problem~\eqref{eq:solutionset}. 
As an example of the setup above, in the case of two-dimensional 
	computed tomography (CT) problems, one can usually model X-Ray
	 detectors using an impulse response $h\in\Ell_q(\reals)$ for some 
	 $q\in(1,\infty)$, so that the $M$
	 measurements at angles $\lbrace\theta_m\rbrace_{m=1}^M\subset[0,\pi)$ 
	 and offsets $\lbrace t_m\rbrace_{m=1}^M\subset \reals$ are given by 
	$\nu_m = h(t_m-\langle \cdot, \boldsymbol{\theta}_m\rangle)$,
	where $\boldsymbol{\theta}_m = [\cos(\theta_m),\sin(\theta_m)]$.
	Then, we have that $\nu_m\in\Ell_q(\Omega)=\Xdual$ for any bounded 
	$\Omega\subset\reals^2$, and thus $\X=\Ell_p(\Omega)$ with $p=1/(1-1/q)$.
	Then, the regularizer in \eqref{eq:solutionset} can be chosen as the 
	$\Ell_p$ regularizer.

The main object of study in this paper, the
optimization problem \eqref{eq:synthesis_form}, corresponds to the setup
above with
\begin{equation}\label{eq:lpreg}
    \X = (\reals^N,\Vert\cdot\Vert_p)\mbox{, }\psi(x)=\lambda x^p
    \mbox{, and } \lambda\in\nnreals\,.
\end{equation}
Theorem~\ref{th:rep} characterizes the solutions of \eqref{eq:solutionset} in
full generality using the duality map of the dual space $\Xdual$.
In this setting, the duality map is the nonlinear map
$\dualitymapdual\colon\Xdual\to\X$ that generalizes
the concept of parallel vectors to Banach spaces (see
Proposition~\ref{def:duality_bound} and Definition~\ref{def:duality_map}).

\begin{theorem}[Representer Theorem for Inverse Problems \cite{Unser2021}]
    \label{th:rep}
    For a reflexive and strictly convex pair $(\X,\Xdual)$ of Banach spaces
    with duality map $\dualitymapdual\colon\Xdual\to\X$,
    the variational problem \eqref{eq:solutionset} has the unique solution
    \begin{equation} \label{eq:representer_theorem}
        \f_\meas = \dualitymapdual\!\left( \nu_\meas \right)
                   \mbox{, with }
                   \nu_\meas = \sum_{m=1}^{M} \rcoeffs_{\meas}[m]\nu_m,
    \end{equation}
    for a unique coefficient vector $\rcoeffs_\meas \in \reals^M$.
\end{theorem}

\begin{figure}
    \begin{tikzpicture}

\definecolor{darkgray176}{RGB}{176,176,176}
\definecolor{darkturquoise0191191}{RGB}{0,191,191}
\definecolor{goldenrod1911910}{RGB}{191,191,0}
\definecolor{lightgray204}{RGB}{204,204,204}

\begin{axis}[
width=.5\textwidth,
height=.4\textwidth,
legend cell align={left},
legend columns=3,
legend style={
  fill opacity=0.8,
  draw opacity=1,
  text opacity=1,
  at={(0.5,0.15)},
  anchor=south,
  draw=lightgray204
},
tick align=outside,
tick pos=left,
x grid style={darkgray176},
xmin=-0.25, xmax=9.25,
xtick style={color=black},
y grid style={darkgray176},
y label style={at={(0,1)}, rotate=-90},
ymin=-7.88193357647592, ymax=7.88193357647592,
ytick style={color=black}
]
\addplot [very thin, black, dashed, forget plot]
table {%
0 0
10 0
};
\addplot [semithick, blue, mark=*, mark size=3, mark options={solid}, only marks]
table {%
0 0.441227486885041
1 -0.330870151894088
2 2.43077118700778
3 -0.252092129603077
4 0.109609841578183
5 1.58248111706156
6 -0.909232404856242
7 -0.591636657930288
8 0.187603225837035
9 -0.329869957779359
};
\addlegendentry{$\nu_0$}
\addplot [semithick, black, dash pattern=on 1pt off 3pt on 3pt off 3pt]
table {%
0 7.16539416043266
1 7.16539416043266
2 7.16539416043266
3 7.16539416043266
4 7.16539416043266
5 7.16539416043266
6 7.16539416043266
7 7.16539416043266
8 7.16539416043266
9 7.16539416043266
};
\addlegendentry{$\pm|| \nu_0 ||_{\ell_1}$}
\addplot [semithick, black, dash pattern=on 1pt off 3pt on 3pt off 3pt, forget plot]
table {%
0 -7.16539416043266
1 -7.16539416043266
2 -7.16539416043266
3 -7.16539416043266
4 -7.16539416043266
5 -7.16539416043266
6 -7.16539416043266
7 -7.16539416043266
8 -7.16539416043266
9 -7.16539416043266
};
\addplot [semithick, red, opacity=0.8, mark=diamond*, mark size=3, mark
options={solid}, only marks]
table {%
0 0.441227486885041
1 -0.330870151894088
2 2.43077118700778
3 -0.252092129603077
4 0.109609841578183
5 1.58248111706156
6 -0.909232404856242
7 -0.591636657930288
8 0.187603225837035
9 -0.329869957779359
};
\addlegendentry{$p = 2$}
\addplot [semithick, darkturquoise0191191, opacity=0.8, mark=triangle*, mark size=3, mark options={solid}, only marks]
table {%
0 1.89492513698935e-74
1 -5.98542755967762e-87
2 2.43077118700778
3 -9.27395117257822e-99
4 6.25591879624563e-135
5 5.56147191074121e-19
6 -4.77338701038791e-43
7 -1.03954575892449e-61
8 1.36573449096039e-111
9 -4.42193735179664e-87
};
\addlegendentry{$p = 1.01$}
\addplot [semithick, goldenrod1911910, opacity=0.8, mark=triangle*, mark size=3, mark options={solid,rotate=180}, only marks]
table {%
0 7.16539416043266
1 -7.16539416043266
2 7.16539416043266
3 -7.16539416043266
4 7.16539416043266
5 7.16539416043266
6 -7.16539416043266
7 -7.16539416043266
8 7.16539416043266
9 -7.16539416043266
};
\addlegendentry{$p = \infty$}
\end{axis}

\end{tikzpicture}
    \caption{Application of different duality maps $\dualitymapdual$
            to a vector $\nu_0\in\Xdual$ when
            $\X=(\reals^{10},\Vert\cdot\Vert_p) \subset \ell_p$ for
            different values of $p$. \label{fig:simple_duality_map}}
\end{figure}

\begin{figure*}
    \centering
    \begin{tikzpicture}
    \begin{axis}[
        width=.34\textwidth,
        height=0.34\textwidth,
        xlabel={$\dcoeffs_1$},
        ylabel={$\dcoeffs_2$},
        zlabel={$\dcoeffs_3$},
        title={$p=1.25$},
        x label style={font=\scriptsize,yshift=0.4cm},
        y label style={font=\scriptsize,yshift=0.35cm,xshift=-0.1},
        z label style={font=\scriptsize,yshift=-0.4cm,xshift=0.1cm},
        ]
        \addplot3+ [opacity=0.35,
        only marks, scatter,
        scatter src=explicit,
        ] table[meta=c] {figs/duality_map_p_1.25.txt};
    \end{axis}
\end{tikzpicture}
\begin{tikzpicture}
    \begin{axis}[
        width=.34\textwidth,
        height=0.34\textwidth,
        xlabel={$\dcoeffs_1$},
        ylabel={$\dcoeffs_2$},
        zlabel={$\dcoeffs_3$},
        title={$p=1.5$},
        x label style={font=\scriptsize,yshift=0.4cm},
        y label style={font=\scriptsize,yshift=0.35cm,xshift=-0.1},
        z label style={font=\scriptsize,yshift=-0.4cm,xshift=0.1cm},
        ]
        \addplot3+ [opacity=0.35,
        only marks, scatter,
        scatter src=explicit,
        ] table[meta=c] {figs/duality_map_p_1.5.txt};
    \end{axis}
\end{tikzpicture}
\begin{tikzpicture}
    \begin{axis}[
        width=.34\textwidth,
        height=0.34\textwidth,
        xlabel={$\dcoeffs_1$},
        ylabel={$\dcoeffs_2$},
        zlabel={$\dcoeffs_3$},
        title={$p=1.75$},
        x label style={font=\scriptsize,yshift=0.4cm},
        y label style={font=\scriptsize,yshift=0.35cm,xshift=-0.1},
        z label style={font=\scriptsize,yshift=-0.4cm,xshift=0.1cm},
        ]
        \addplot3+ [opacity=0.35,
        only marks, scatter,
        scatter src=explicit,
        ] table[meta=c] {figs/duality_map_p_1.75.txt};
    \end{axis}
\end{tikzpicture}
    \caption{Dependence of $\dcoeffs=\dualitymapdual(\nu)$ on the vector of coefficients
    $\rcoeffs$, with $\nu=\rcoeffs[1] \nu_1 + \rcoeffs[2] \nu_2$. The norm of the vector of coefficients
    is encoded in the intensities. Here, $\X=(\reals^3,\Vert\cdot\Vert_p)$ for
    $p\in\lbrace1.25,1.5,1.75\rbrace$, the measurement operator
    $\measop:\reals^3 \to \reals^2$ is composed of two $\ell_2$-normalized
    random measurement vectors $\nu_1$ and $\nu_2$, and a regular ($32\times 32$)
    grid of $\rcoeffs$ in $[-1,1]^2$ is explored.
\label{fig:manifold_vs_a}}
\end{figure*}

Theorem~\ref{th:rep} reveals several favorable properties of the variational
approach to inverse problems. First, it guarantees that \eqref{eq:solutionset}
always has a well-defined, unique solution.
Further, \eqref{eq:representer_theorem} transforms the search for
$\f_\meas\in\X$ (in the case of \eqref{eq:synthesis_form}, of dimension
$N\geq M$ and, in general, possibly infinite-dimensional) into a
finite-dimensional search for $\rcoeffs_\meas\in\reals^M$, a vector of the
same dimension as the measurements. Prior information is injected in the
solution by means of the duality map $\dualitymapdual$ given by the chosen
regularization---through the choice $(\X,\Xdual)$---which maps
$\nu_{\meas}\mapsto \f_\meas$.
The regularizing effect of the duality map can be seen in
Figure~\ref{fig:simple_duality_map}. The general result of
Theorem~\ref{th:rep} fits within a family of representer theorems for
variational problems
\cite{Zhang2012,Unser2016,Flinth2018,Boyer2019,Unser2019,Bredies2020,Unser2020a
,Unser2021,BarDeVRosVig2021}.

For the variational problem \eqref{eq:synthesis_form}, Theorem~\ref{th:rep}
guarantees that the reconstruction map $\solutionmap\colon\reals^M\to\X$
is well defined and that it is the composition of an unknown map
$\meas\mapsto \nu_{\meas}$ with the duality map $\dualitymapdual$. Thereby,
the study of $\dualitymapdual$ offers a first, more intuitive, approximation
to the problem.
Fortunately, its expression is known in closed-form for $\X=\ell_p$ (see
Proposition~\ref{prop:duality_map_ellq}).
Figure~\ref{fig:manifold_vs_a} illustrates
the duality map with small computational examples for
$p\in\lbrace1.25,1.5,1.75\rbrace$, where we control the
coefficient vector $\rcoeffs_\meas\in\reals^2$ and depict the resulting
solutions $\dcoeffs_\meas\in\reals^3$. These suggest that stability highly
depends on $p$, with rougher landscapes as $p$ approaches $1$ and smoother
ones as it approaches $2$.

As a prelude to Section~\ref{sec:unicity-and-stability}, where we
characterize the robustness of image reconstruction using $\ell_p$
regularization for $p\in(1,\infty)$, we now discuss the two topical examples of
$\ell_2$ regularization and sparsity-promoting $\ell_1$ regularization.
Although known for the most part, the results will
set the context for the later, more general results for $p\in(1,\infty)$.

    \subsection{\texorpdfstring{$\ell_2$}{l2} Regularization and Hilbert
                Spaces} \label{sec:l2}
        % !TEX root = ../main.tex
The classical example of variational regularization is
$\ell_2$ regularization, as in \eqref{eq:lpreg} with $p=2$.
As it turns out, the analysis of the stability of the solution is not
specific to the finite-dimensional case where $\dcoeffs\in\reals^N$. It
can be transposed into the more general space of finite-energy discrete signals
$\f\in\ell_2$. In fact, because $\ell_2$ is a Hilbert space (a Banach space
with an inner product, see Definition~\ref{def:hilbert}), $\ell_2$
regularization can be analyzed in the
broader setting of Tikhonov regularization in Hilbert spaces. The following
analysis is valid for every Hilbert space $\H$, including $\Ell_2(\Omega)$ and
other
Hilbert/Sobolev function spaces. In the context of Theorem~\ref{th:rep}, then,
this
corresponds to choosing $\X=\H$. Because Hilbert spaces are strictly
convex and reflexive, Theorem~\ref{th:rep} applies. The duality map in Hilbert
spaces with
the choices above corresponds to the Riesz map
$\Rieszmap\colon\H'\rightarrow\H$ (see Definition~\ref{def:riesz_map} and the
subsequent discussion). The Riesz map is linear. Thus, it holds that
\begin{equation} \label{eq:l2solution}
    \f_\meas = \sum_{m=1}^{M} \rcoeffs_{\meas}[m] \varphi_m, \mbox{ where }
    \varphi_{m} = \Rieszmap\lbrace\nu_{m}\rbrace
\end{equation}
for a unique vector of coefficients $\rcoeffs_{\meas}\in\reals^M$.
Using that $\Vert\f\Vert^2_\H = \langle \f,\f \rangle_\H$ for any $\f\in\H$
(see Definition~\ref{def:hilbert}), and that
$\langle \nu_m, \varphi_n \rangle_{\H'\times\H} = \langle \varphi_m, \varphi_n\rangle_\H$
(see Definition~\ref{def:riesz_map}), where $\langle\cdot,\cdot\rangle_{\H}$ is the inner
product in $\H$, we obtain that
\begin{equation} \label{eq:finite-rewrite-l2}
    \Vert \f_\meas \Vert^2_{\H} = \rcoeffs_{\meas}^{\T}\matrix\rcoeffs_{\meas}
    \mbox{, and }\measop(\f_{\meas}) = \matrix \rcoeffs_{\meas}\,,
\end{equation}
with
$$
    \matrix = \left( \langle \varphi_m, \varphi_n\rangle_{\H}
              \right)_{n,m\in\lbrace1,\dots,M\rbrace}\in
    \reals^{M \times M}.
$$

Using \eqref{eq:l2solution} and \eqref{eq:finite-rewrite-l2}, we then
write~\eqref{eq:solutionset} as the finite-dimensional optimization
problem
\begin{equation}\label{eq:l2case}
    \min_{\rcoeffs \in \reals^M}\bigl\lbrace
        E\left( \meas,  \matrix\rcoeffs \right) +
        \lambda \,\rcoeffs^{\mathrm{T}} \matrix \rcoeffs
    \bigr\rbrace.
\end{equation}
For the specific case of least squares, where $E$ is chosen as
\begin{equation} \label{eq:LSdatafidelity}
    E(\meas_1,\meas_2) = \frac12 \Vert \meas_1 - \meas_2\Vert_2^2,
\end{equation}
this results in a fully quadratic problem on the coefficients $\rcoeffs$,
with closed-form solution
\begin{equation}\label{eq:closed_form_ls-tikhonov}
    \rcoeffs_\meas = \bigl(\matrix^{\T}\matrix +
                   2\lambda\matrix^{\T}\bigr)^{-1} \matrix^{\T} \meas
               = (\matrix + 2 \lambda \Id)^{-1} \meas,
\end{equation}
where we took advantage of the property that $\matrix$ is Hermitian by
construction and full-rank due to the linear independence of the measurement
functionals $\nu_m$. This allows us to characterize stability as in
Proposition~\ref{prop:Lipschitzl2}.

\begin{figure}
    \input{figs/l2bound_expresults}

    \vspace{-20pt}

    \caption{Lipschitz constant in \eqref{eq:l2tightest}
        compared to experimental values of
        $\Vert\dcoeffs_{\meas_1} - \f_{\meas_2}\Vert_2 /
        \Vert\meas_1-\meas_2\Vert_2$ for $\ell_2$-regularized
        least-squares problems with $\lambda=1$ and
        $\tilde{\measop}=\sigma\measop$, where
        $\measop\colon\reals^3\to\reals^2$ has random-normal entries.
\label{fig:l2_experimental_results}}
\end{figure}

\begin{proposition}[Lipschitz Continuity of Tikhonov-Regularized Least Squares]
\label{prop:Lipschitzl2}
    Consider the $\ell_2$-regularized least-squares optimization
    problem
    \begin{equation} \label{eq:l2ls}
        \min_{\f\in\ell_2} \Bigl\lbrace
            \frac12 \left\Vert \meas - \measop(\f)\right\Vert_2^2
            + \lambda \left\Vert \f \right\Vert^2_{\H}\Bigr\rbrace.
    \end{equation}
    Moreover, consider two measurements $\meas_1,\meas_2\in\reals^M$ and their
    associated solutions of \eqref{eq:l2ls},
    $\f_{\meas_1},f_{\meas_2}\in\ell_2$.
    Then, one has that
    \begin{equation} \label{eq:l2tightest}
        \Vert \f_{\meas_1} - \f_{\meas_2} \Vert_{\H} \leq
            \max_{m\in\lbrace 1,2,\dots, M\rbrace}
                        \frac{\sqrt{\sigma_m}}{\sigma_m + 2\lambda}
            \Vert \meas_1 - \meas_2 \Vert_2
    \end{equation}
    and the reconstruction map $\solutionmap\colon\meas\mapsto\f_\meas$ is
    Lipschitz continuous. Here, $\lbrace\sigma_m\rbrace_{m=1}^{M}$ is the set
    of eigenvalues of the matrix $\matrix$.
\end{proposition}
\begin{proof}
    Consider \eqref{eq:closed_form_ls-tikhonov} and
    \eqref{eq:finite-rewrite-l2}. Then,
    \begin{align*}
    \Vert \f_{\meas_1} - \f_{\meas_2} \Vert_{\H}^2 &=
        (\dcoeffs_{\meas_1}-\dcoeffs_{\meas_2})^{T}\matrix\,
        (\dcoeffs_{\meas_1}-\dcoeffs_{\meas_2}) \\
        &= (\meas_1-\meas_2)^{\T} \mathbf{P}
            \frac{\boldsymbol{\Lambda}}{(\boldsymbol{\Lambda} + 2\lambda\Id)^2}
            \mathbf{P}^{T} (\meas_1 - \meas_2) \\
        & \leq
        \max_{m\in\lbrace 1,\dots, M\rbrace}
                        \frac{\sigma_m}{(\sigma_m + 2\lambda)^2}
            \Vert \meas_1 - \meas_2 \Vert_2^2\,, \tag{15'}
    \end{align*}
    where $\mathbf{P}$ is the orthogonal matrix of eigenvectors of $\matrix$
    such that $\matrix = \mathbf{P}\boldsymbol{\Lambda} \mathbf{P}^{\T}$
    and $\boldsymbol{\Lambda}$ is the diagonal matrix containing the
    eigenvalues of $\matrix$.
\end{proof}

Of particular interest is how the bound~\eqref{eq:l2tightest} scales with
respect to the measurement functionals $\nu_m$ and their Gram matrix $\matrix$
for a given regularization parameter $\lambda$.
For instance, if we consider a measurement operator
$\tilde{\measop} = \sigma \measop$, we see that the Lipschitz constant
in~\eqref{eq:l2tightest} decays as
$1/\min_{m}\lbrace\sigma\sqrt{\sigma_{m}}\rbrace$ for $\sigma\rightarrow\infty$.
In other words, stability is ultimately regulated by those changes to which
$\tilde{\measop}$ is least sensitive. The behavior of the bound with respect
to $\sigma$ as well as empirical results are portrayed in
Figure~\ref{fig:l2_experimental_results}, together with the expected asymptotic behavior.

As expected, in \eqref{eq:l2tightest} we see that increasing the regularization parameter $\lambda$ will result in
more stable solution maps. However, doing so will also increase the bias of the 
resulting algorithm---c.f. \eqref{eq:closed_form_ls-tikhonov}---negatively affecting performance. Our results throughout the 
paper aim to compare the stability of algorithms once all parameters have been selected
to obtain the best achievable performance.

To summarize, $\ell_2$ regularization (and, in general, Tikhonov
regularization in any Hilbert space) leads to a unique solution, with a
solution map that is globally Lipschitz continuous.
Although Proposition~\ref{prop:Lipschitzl2} only covers
least-squares problems, we shall see
in Section~\ref{sec:unicity-and-stability} that this remains true for any
other strictly convex data-fidelity term.

    \subsection{\texorpdfstring{$\ell_1$}{l1} Regularization and
                Sparsity} \label{sec:l1}
        % !TEX root = ../main.tex
Variational image reconstruction driven by sparsity-promoting regularization
using the $\ell_1$ norm is supported by the theory known as compressed
sensing~\cite{Candes2006,Donoho2006}. The optimization problem
\begin{equation}\label{eq:constrained_l1}
    \min_{\dcoeffs\in\reals^N}\Vert \dcoeffs \Vert_1 \mbox{
subject to }
    \Vert \matrixA \dcoeffs - \meas \Vert_2 \leq \sigma
\end{equation}
has for solution the sparsest vector within the constraint set
$\mathcal{C}_\sigma = \lbrace \dcoeffs\in\reals^N : \Vert \matrixA \dcoeffs -
\meas \Vert_2\leq \sigma\rbrace$, provided $\matrixA$ fulfills some rather
strict conditions, namely, the restricted-isometry property. The minimization
in \eqref{eq:constrained_l1} is portrayed in Figure~\ref{fig:l1unique}, where
the radii of the $\ell_2$- and $\ell_1$-norm balls are increased until they
meet the boundary of $\mathcal{C}_\sigma$. In that example, the $\ell_1$-norm
minimization does indeed lead to a sparse solution. Although this analysis
relies on the constrained formulation in \eqref{eq:constrained_l1}, the effect
of the $\ell_1$-norm on the set of solutions of the corresponding regularized
formulation is effectively the same~\cite[Remark~3.3]{McCann2019}. This
regularized formulation corresponds to \eqref{eq:lpreg} when $p=1$.

\begin{figure}
    \centering
    \begin{tikzpicture}[scale=1.5]
    % Draw axis
    \draw[-latex] (-2,0) -- (3,0) node[anchor=north] {$\dcoeffs_1$};
    \draw[-latex] (0,-1.5) -- (0,1.9) node [anchor=east] {$\dcoeffs_2$};
    % Draw unit balls
    % l2 ball
    \draw[thick,blue,dashed] (0,0) circle (1);
    % l1 ball
    \def\s{1.15}
    \draw[thick,red,dashdotted] (-\s,0) -- (0,\s) -- (\s,0) -- (0,-\s) -- cycle;
    % Draw measurement line
    \draw[->,thick] (30:-2) -- (30:3) node [anchor=south] {$\matrixA
\dcoeffs$};
    % Draw measurement and constant L_2 norm lines
    \draw (30:1.5)--+ (120:-0.1) --+ (120:0.1) node
        [anchor=south] {$\meas$};
    \draw[dotted] (30:1) --+ (120:1.5) --+ (120:-2);
    \draw[dotted] (30:2) --+ (120:1.5) --+ (120:-2);
    \node at (60:1.75) {$\mathcal{C}_\sigma$};
    % Draw l2 and l1 solutions
    \node[blue,opacity=0.5] at (30:1) {\Large $\bullet$};
    \node[red,opacity=0.5] at (\s,0) {\small $\bigstar$};
\end{tikzpicture}
    \caption{Pictorial representation of the solution to
        $\min_{\dcoeffs\in\reals^N} \Vert \dcoeffs \Vert_{p}$ subject to
        $\|\matrixA \dcoeffs-\meas\|_2\leq\sigma$ for $p\in\lbrace1,2\rbrace$.
        Here, $\matrixA\in\reals^{1\times2}$ and
        $\dcoeffs=(\dcoeffs_1,\dcoeffs_2)\in\reals^2$. The solution is unique
        both for $p=1$ and $p=2$, while the $\ell_1$ penalty leads to the
        sparse solution, indicated by a star. \label{fig:l1unique}}
\end{figure}

The same restricted-isometry property that guarantees a unique,
sparse solution to \eqref{eq:constrained_l1} also provides the Lipschitz
stability~\cite{Candes2006} of that solution with respect to
variations in the measurements. In particular, when $\matrixA$ is
composed of linearly independent measurement functionals, as assumed in
Theorem~\ref{th:rep}, it can be shown that there exists a constant
$K_\mathrm{s}\in\nnreals$ such that
\begin{equation} \label{eq:CSLipschitz}
    \Vert \dcoeffs_{\meas_1} - \dcoeffs_{\meas_2} \Vert_2 \leq K_{\mathrm{s}}
    \Vert \meas_1 - \meas_2 \Vert_{2}.
\end{equation}
The caveat, however, is that it is challenging (or impossible) to obtain or
bound  the constant $K_\mathrm{s}$ for specific practical problems.
Because compressed sensing is formulated for finite-dimensional spaces, norm
equivalence, together with \eqref{eq:CSLipschitz}, imply Lipschitz stability,
as initially described in \eqref{eq:local_Holder} ($\beta=1$ and $Y=\reals^M$). In particular, the fact that
$\Vert \dcoeffs \Vert_{1}\leq K_{1,2} \Vert \dcoeffs \Vert_{2}$ for any
$\dcoeffs\in\reals^N$ and some fixed $K_{1,2}\in\nnreals$ implies that
$K=K_{1,2}K_\mathrm{s}$.

These attractive theoretical results, however, cannot be applied
to many sparsity-based image-reconstruction algorithms, particularly to those
that fail to operate within the regime of the restricted-isometry
property. When this happens, one may be faced with an infinite
number of solutions of \eqref{eq:constrained_l1}, as portrayed in
Figure~\ref{fig:l1nonunique}.
Then, the corresponding regularized problem has to be characterized by a
version of the representer theorem that is more general than the
one we presented in Theorem~\ref{th:rep}~\cite[Section~4]{Unser2021}. Indeed,
the extended version contemplates Banach spaces that are not strictly convex
and the solution set is shown to be the convex hull of a number of sparse
extreme points. Note that, under slightly more general conditions than the
restricted isometry property, weaker forms of stability than the ones we are
interested in can still be obtained for $\ell_1$
regularization~\cite{GraHalSch11,FleHofVes16}.

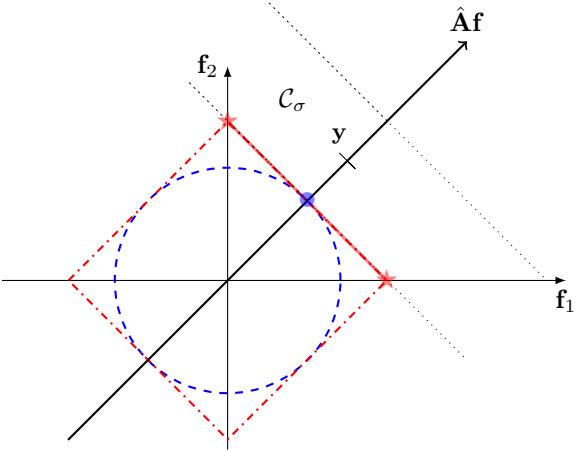
\begin{figure}
    \centering
    \begin{tikzpicture}[scale=1.5]
    % Draw axis
    \draw[-latex] (-2,0) -- (3,0) node[anchor=north] {$\dcoeffs_1$};
    \draw[-latex] (0,-1.5) -- (0,1.9) node [anchor=east] {$\dcoeffs_2$};
    % Draw unit balls
    % l2 ball
    \draw[thick,blue,dashed] (0,0) circle (1);
    % l1 ball
    \def\s{1.41}
    \draw[thick,red,dashdotted] (-\s,0) -- (0,\s) -- (\s,0) -- (0,-\s) -- cycle;
    \draw[very thick, red,opacity=0.5] (\s,0) -- (0,\s);
    % Draw measurement line
    \draw[->,thick] (45:-2) -- (45:3) node [anchor=south]
        {$\hat{\matrixA} \dcoeffs$};
    % Draw measurement and constant L_2 norm lines
    \draw (45:1.5)--+ (135:-0.1) --+ (135:0.1) node [anchor=south] {\small $\meas$};
    \draw[dotted] (45:1) --+ (135:1.5) --+ (135:-2);
    \draw[dotted] (45:2) --+ (135:1.5) --+ (135:-2);
    \node at (70:1.7) {$\mathcal{C}_\sigma$};
    % Draw l2 and l1 solutions
    \node[blue,opacity=0.5] at (45:1) {\Large $\bullet$};
    \node[red,opacity=0.5] at (\s,0) {\small $\bigstar$};
    \node[red,opacity=0.5] at (0,\s) {\small $\bigstar$};
\end{tikzpicture}
    \caption{Same representation as in Figure~\ref{fig:l1unique}, but for
        another linear operator. Here, the $\ell_1$
        penalty leads to a problem with infinitely many solutions. In
        particular, the solution set contains an entire edge of the
        $\ell_1$-norm unit ball: the convex hull of
        $\lbrace (0,1),(1,0)\rbrace$. This generalizes well
        to the infinite-dimensional sparsity-promoting regularization
        setting, see~\cite{Unser2016}.
    \label{fig:l1nonunique}}
\end{figure}

\section{Stability of Solutions} \label{sec:unicity-and-stability}
    % !TEX root = ../main.tex
In this section, we present our results for the general case of $\ell_p$
regularization with $p\in(1,\infty)$. Because, as in the case of $\ell_2$
regularization in Section~\ref{sec:l2}, the analysis of the stability of the
solution is not particular to the finite-dimensional setting, we directly
expose it for a (potentially infinitely-supported) discrete signal
$\f\in\ell_p$. We study the problem
\begin{equation} \label{eq:ellpproblem}
    \min_{\f\in\ell_p} \bigl\lbrace E\left(\meas,\measop(\f)\right)
                                    + \lambda \Vert \f \Vert^p_{\ell_p}
                       \bigr\rbrace
\end{equation}
with $\lambda\in\nnreals$, under the following assumptions:
\begin{enumerate}[({A}1)]
    \item the data-fidelity term
        $E\colon\reals^M\times\reals^M\to\reals_+$ is
        lower-semicontinuous
        and strictly convex in its second argument; \label{ass:niceE}
    \item the data-fidelity term $E$ is differentiable in its second argument, and
        the gradient
        $\nabla_f\left\lbrace E(\meas,\measop(\f)) \right\rbrace$ is
        Lipschitz continuous with respect to changes in the measurements
        $\meas$ with constant $\LipConstant_p$ for any $\f$;
        \label{ass:LipschitzdiffE}
    \item the measurement operator $\measop$ is composed of $M$ linearly
        independent functionals
$\lbrace\nu_m\rbrace_{m=1}^{M}$ as in
        \eqref{eq:measop}.
        \label{ass:indepmeasop}
\end{enumerate}
We now discuss our main results (Theorems~\ref{th:lp2toinf}~and~\ref{th:lp1to2}).
The proofs are deferred to their more general cases
(Theorems~\ref{th:Lp2toinf}~and~\ref{th:Lp1to2}), which cover the case of
$\Ell_p(\Omega)$ function spaces.

\begin{table}
    \caption{Continuity bounds obtained on the
        stability of solutions to $\ell_p$-regularized inverse problems.
        \label{tab:bounds}}
    \begin{tabular}{r|m{.35\textwidth}}
    \hline\hline
        $p\in(1,2)$ &
        $$\Vert \f_{\meas_1} - \f_{\meas_2} \Vert_{\ell_p} \leq
            \frac{(2\,r_p(Y))^{2-p}\LipConstant_p}{\lambda p(p-1)}
            \Vert \meas_1 - \meas_2 \Vert_2$$ \\ \hline
        $p\in[2,\infty)$ &
        $$\Vert \f_{\meas_1} - \f_{\meas_2} \Vert_{\ell_p} \leq
            \left(
                \frac{2^{p-2}\LipConstant_p}{\lambda p}
            \right)^{\frac{1}{p-1}}
            \Vert \meas_1 - \meas_2 \Vert_2^{\frac{1}{p-1}}$$ \\
    \hline\hline
    \end{tabular}
\end{table}

\subsection{Hölder Continuity for \texorpdfstring{$p\in[2,\infty)$}{p grater or
    equal than 2}}

    The case of $p=2$ has been discussed for least-squares problems
    in Section~\ref{sec:l2}. There, we proved Lipschitz continuity of the
    solution using a Hilbert-space analysis (see
    Proposition~\ref{prop:Lipschitzl2}). Our result here extends this
    claim for $p=2$ to any $E$ that fulfills (A\ref{ass:niceE}) and
    (A\ref{ass:LipschitzdiffE}).

    \begin{theorem}[Hölder Continuity of $\ell_p$-Regularized Linear Inverse
        Problems, $p\in[2,\infty)$] \label{th:lp2toinf}
        Consider the variational problem \eqref{eq:ellpproblem} with
        $p\in[2,\infty)$ and assume that
        (A\ref{ass:niceE})--(A\ref{ass:indepmeasop}) are fulfilled. Then,
        for any given measurement $\meas\in\reals^M$, \eqref{eq:ellpproblem}
        has a unique solution $\f_\meas\in\ell_p$. Moreover, it holds that
        \begin{equation} \label{eq:lp2toinf}
            \Vert \f_{\meas_1} - \f_{\meas_2} \Vert_{\ell_p} \leq
            \left(
                \frac{2^{p-2}\LipConstant_p}{\lambda p}
            \right)^{\frac{1}{p-1}}
            \Vert \meas_1 - \meas_2 \Vert_2^{\frac{1}{p-1}},
        \end{equation}
        for any two $\meas_1,\meas_2\in\reals^M$.
    \end{theorem}

    The exponent $\beta =  1/(p-1)$ in \eqref{eq:lp2toinf} characterizes the
    continuity bound of Theorem~\ref{th:lp2toinf}. In particular, this
    exponent takes unit value for $p=2$, which indeed leads to Lipschitz
    continuity, \emph{cf.} \eqref{eq:local_Holder} with $\beta=1$ and $Y=\reals^M$. For
    $p>2$, the result is weaker because, in the low-noise
    regime where $\Vert\meas_1-\meas_2\Vert_2\rightarrow 0$, the
    bound on $\Vert\f_{\meas_1} - \f_{\meas_2}\Vert_{\ell_p}$ gets larger
    as $p$ increases.
    The change of the type of continuity result for different values of $p$ is
    an unexpected result. One might think that this is caused by the mismatch
    between the norms used in either side of \eqref{eq:lp2toinf}. However,
    because all norms are equivalent in the finite-dimensional space of measurements $\reals^M$,
    the exponent in \eqref{eq:lp2toinf} holds for every other norm in
    $\reals^M$, and the choice of any other specific norm would simply be absorbed by the
    preceding constant.

    The term $\LipConstant_p$ that appears in the leading constant of the
    bound characterizes the role of the data-fidelity term $E$ on the
    stability of the solution. Most importantly, and as opposed to what
    happens in \eqref{eq:CSLipschitz} with compressed sensing,
    all terms in the bound of Theorem~\ref{th:lp2toinf} can be computed or
    bounded in practical applications. In particular we evaluate our result
    in Section~\ref{sec:stability_l2lp} in the case of $\ell_p$-regularized
    least-squares problems, where $\LipConstant_p$ is tied to the norm of
    the measurement operator $\measop$.

\subsection{Local Lipschitz Continuity for
    \texorpdfstring{$p\in(1,2)$}{p between 1 and 2}}

    Now, we turn to the case $p\in(1,2)$ which describes the continuous
    bridge between sparsity-promoting regularization and Tikhonov
    regularization. Our mathematical treatment requires one additional
    assumption
    \begin{enumerate}[({A}1)]
    \setcounter{enumi}{3}
        \item There is at least one $\tilde\f$ where the
            data-fidelity term $E(\meas,\measop(\tilde\f))$ is continuous
            with respect to its first argument $\meas$.
            \label{ass:contE}
    \end{enumerate}

    \begin{theorem}[Local Lipschitz Continuity of $\ell_p$-Regularized Linear
        Inverse Problems, $p\in(1,2)$] \label{th:lp1to2}
        Consider the variational problem \eqref{eq:ellpproblem} with
        $p\in(1,2)$ and assume that
        (A\ref{ass:niceE})--(A\ref{ass:contE}) are fulfilled. Then,
        for any given measurement $\meas\in\reals^M$, \eqref{eq:ellpproblem}
        has a unique solution $\f_\meas\in\ell_p$. Moreover, for any
        compact set $Y\subset\reals^M$, there is a constant
        $r_p(Y)\in\nnreals$ such that
        $$
            \Vert \f_{\meas_1} - \f_{\meas_2} \Vert_{\ell_p} \leq
            \frac{(2\,r_p(Y))^{2-p}\LipConstant_p}{\lambda p(p-1)}
            \Vert \meas_1 - \meas_2 \Vert_2,
        $$
        for any two measurements $\meas_1,\meas_2\in Y$.
    \end{theorem}

    We thus recover again Lipschitz continuity for $p\rightarrow 2$, for which
    the bound coincides with the one in Theorem~\ref{th:lp2toinf}.
    For $p\in(1,2)$, we achieve Lipschitz continuity,
    albeit only locally in the space $\reals^M$ of measurements.
    Assumption (A\ref{ass:contE}) guarantees the existence of
    the constant $r_p(Y)$. We defer to
    Section~\ref{sec:mathematical_formalization} the establishment of an
    easy-to-evaluate upper bound for this constant.

\subsection{\texorpdfstring{$\ell_p$}{lp}-Regularized Least Squares}
    \label{sec:stability_l2lp}

    Regularized least squares, where the data-fidelity term $E$ is given by
    \eqref{eq:LSdatafidelity}, is popular in many applications. In this
    section, we choose $\ell_p$-regularized least squares to illustrate our
    findings expressed in Theorems~\ref{th:lp2toinf} and~\ref{th:lp1to2}. For
    this scenario, (A\ref{ass:niceE}) and (A\ref{ass:contE}) hold due to the
    the properties of the $\ell_2$ norm in finite-dimensional spaces.
    For (A\ref{ass:LipschitzdiffE}), we have that $E$ is differentiable with
    gradient
    $$
        E'(\meas) =
        \boldsymbol{\nabla}_{\f}\left\lbrace E(\meas, \measop(\f))\right\rbrace
        = \measop^* (\measop \f - \meas),
    $$
    where $\measop^*\colon\reals^M \to \ell_{q}$ is the
adjoint    of the measurement operator $\measop$, with
$q=(1-1/p)^{-1}$.
    Thus,
    $$
        E'(\meas_1) - E'(\meas_2) = \measop^* (\meas_2-\meas_1)
    $$
    and, therefore,
    $$
        \LipConstant_p = \Vert \measop \Vert_{\LinOp{\ell_p}{\reals^M}}
                        \coloneqq \sup_{\f\in\ell_p}
                            \frac{\Vert\measop(\f)\Vert_{2}}{
                            \Vert\f\Vert_{\ell_p}},
    $$
    where $\Vert \cdot \Vert_{\LinOp{\ell_p}{\reals^M}}$ is the norm of an
    operator from $\ell_p$ to $\reals^M$.
    For a generic value of $p$, this norm might be challenging to
    compute or approximate. However, if $p$ is between two values for which we
    know how to compute it, we can use the Riesz-Thorin theorem to bound it
    from above. For instance, for $p\in(1,2)$ and $\theta_p=(2-2/p)$, we
    obtain that
    $$
        \Vert\measop\Vert_{\LinOp{\ell_p}{\reals^M}} \leq
        \Vert\measop\Vert_{\LinOp{\ell_1}{\reals^M}}^{1-\theta_p}
        \Vert\measop\Vert_{\LinOp{\ell_2}{\reals^M}}^{\theta_p}.
    $$

    \subsubsection{Revisiting \texorpdfstring{$\ell_2$}{l2}-Regularized Least
        Squares}

        For $\ell_2$, as for any other Hilbert space $\H$, we have that
        $$
            \Vert\measop\Vert_{\LinOp{\ell_2}{\reals^M}}
            = \max_{m\in\lbrace1,\dots,M\rbrace}
                                                    \sqrt{\sigma_m},
        $$
        where $\sigma_m$ is defined with respect to the Gram matrix $\matrix$
        as in Section~\ref{sec:l2}. Therefore, for $p=2$, the bounds in both
        Theorems~\ref{th:lp2toinf} and \ref{th:lp1to2} evaluate to
        \begin{equation} \label{eq:l2loose}
            \Vert \f_{\meas_1} - \f_{\meas_2} \Vert_{\ell_2}
            \leq \max_{m\in\lbrace1,\dots,M\rbrace}
                    \frac{\sqrt{\sigma_m}}{2 \lambda}
            \Vert\meas_1 - \meas_2\Vert_2\,.
        \end{equation}
        This bound is consistent with \eqref{eq:l2tightest}, but
        looser. In particular, it does
        not take into account the assymptotic regime visible in
        Figure~\ref{fig:l2_experimental_results}, where the norm of the
        measurement operator dominates. Instead, it characterizes
        \eqref{eq:ellpproblem} in the strongly regularized
        regime that corresponds to large $\lambda$. We verify this insight
        providing in Figure~\ref{fig:bothbounds} a visual comparison of
        \eqref{eq:l2tightest} and \eqref{eq:l2loose} in terms of $\lambda$.
        This relative behavior is not surprising because
        Theorems~\ref{th:lp2toinf} and~\ref{th:lp1to2} do not exploit the
        linearity of the Riesz map, which leads to the closed-form solution
        \eqref{eq:closed_form_ls-tikhonov} and, thereby, to a more nuanced
        understanding of the effect of the norm of the measurement operator on
        stability. Here, we note that the behavior observed
        in Figure~\ref{fig:bothbounds} is to be expected: increased
        regularization leads to more stable solution maps. However,
        the value of the regularization parameter $\lambda$ still needs
        to be tuned to obtain the least error, not the most stable
        reconstruction map.

        \begin{figure}
            \input{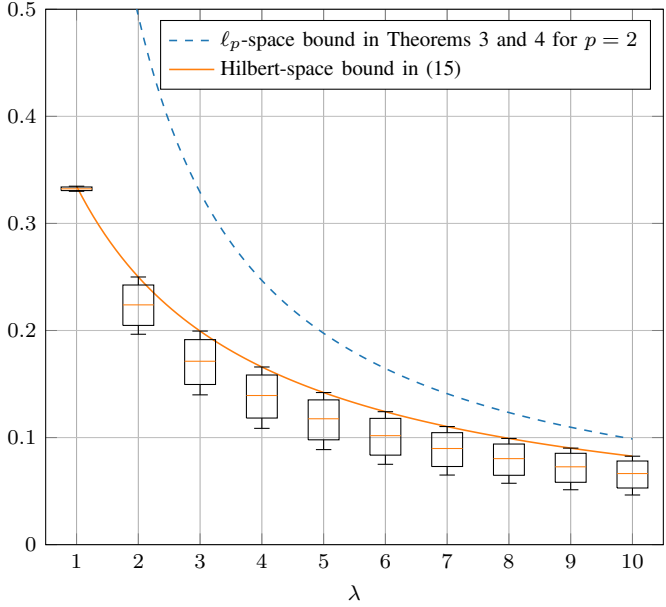}

            \vspace{-20pt}

            \caption{Lipschitz constant derived in \eqref{eq:l2tightest}
                    compared to \eqref{eq:l2loose}, the bound obtained in
                    Theorems~\ref{th:lp2toinf} and~\ref{th:lp1to2} for $p=2$,
                    in terms of the regularization parameter
                    $\lambda\in\nnreals$. The experimental values are as in
                    Figure~\ref{fig:l2_experimental_results}.
                \label{fig:bothbounds}}
        \end{figure}

    \subsubsection{Local Behavior of \texorpdfstring{$\ell_p$}{lp}-Regularized
                   Least Squares for
                    \texorpdfstring{$p\in(1,2)$}{p between 1 and 2}}

        As example of the local behavior ruled by
        Theorem~\ref{th:lp1to2}, we now discuss (A\ref{ass:contE}) in more
        details and exemplify the local constant $r_p(Y)$ when $Y$ is the
        closed ball
        $Y = \lbrace \meas\in \reals^M : \Vert\meas\Vert_2\leq\rho\rbrace$.
        Essentially, $r_p(Y)$ bounds the norm of the optimal solution
        and (A\ref{ass:contE}) provides a (loose) bound
        based on the continuity of the overall cost function at
        $\tilde\f$ (see Section~\ref{sec:mathematical_formalization}).

        Because the least-squares data-fidelity term is continuous for any
        $\f\in\ell_p$, we may choose $\tilde\f=0$ for simplicity. This results
        in $r_p(Y) = (\rho^2 / 2\lambda )^{1/p}$ and, thus,
        Theorem~\ref{th:lp1to2} implies, for $p\in(1,2)$ and for
        $\meas_1,\meas_2\in Y$ as above, that
        $$
        \frac{\Vert \f_{\meas_1} - \f_{\meas_2} \Vert_{\ell_p}}{
                \Vert \meas_1 - \meas_2\Vert_2}
        \leq \left(2^{p-1} \frac{\rho^2}{\lambda}\right)^{\frac{2-p}{p}} \!\!
                \frac{1}{\lambda p(p-1)}
                \Vert \measop \Vert_{\LinOp{\ell_p}{ \reals^M}}.
        $$

\section{Mathematical Formalization} \label{sec:mathematical_formalization}
    % !TEX root = ../main.tex
We adopt a functional formulation to prove the results of
Theorems~\ref{th:lp2toinf} and~\ref{th:lp1to2} in full
generality. The interested reader who would not be
acquainted with the terminology is referred to
Appendix~\ref{app:premath}, which includes a curated
selection of definitions and discussions.

We start by formalizing our assumptions and statements
with respect to the optimization problem
\begin{equation} \label{eq:Ellpproblem}
    \min_{\f\in\Ell_p(\Omega)} \bigl\lbrace E\left(\meas,\measop(\f)\right)
                                    + \lambda \Vert \f \Vert^p_{\Ell_p(\Omega)}
                       \bigr\rbrace,
\end{equation}
which generalizes the analysis to function spaces $\Ell_p(\Omega)$ over a
domain $\Omega\subset\reals^d$. The choice of a countable or finite $\Omega$
equipped with the counting measure particularizes \eqref{eq:Ellpproblem} to
either \eqref{eq:synthesis_form} or \eqref{eq:ellpproblem}, respectively.

Similarly to the relationship between \eqref{eq:image_rec} and
\eqref{eq:synthesis_form}, \eqref{eq:Ellpproblem} can be seen as the synthesis
formulation of the reconstruction problem
$$
    \min_{\tilde\f\in\Ell_{p,\mathrm{L}}(\Omega)} \left\lbrace
                            E\bigl(\meas,\tilde\measop(\tilde\f)\bigr)
                    + \lambda \Vert \mathrm{L}\tilde\f \Vert^p_{\Ell_p(\Omega)}
                       \right\rbrace,
$$
where the regularization operator $\mathrm{L}$ is invertible, $\f = \mathrm{L}\tilde\f$,
and $\measop = \tilde\measop \circ \mathrm{L}^{-1}$. This type of variational
problem has been used, for example, for spline-based interpolation and
approximation
\cite{deBoor1976,MicSmiSweWar1985,Lavery2000,AuqGibNyi2007,BohUns2020},
inverse diffusion~\cite{delAguilaPla2018,delAguilaPla2018a},
and inverse scattering~\cite{LecKazKar13}.

In Sections~\ref{app:abstract} and~\ref{app:MVT}, we present results that are
instrumental to complete our proofs.

\subsection{Abstract Characterization of Stability in Variational Inverse
        Problems} \label{app:abstract}
        % !TEX root = ../main.tex
Let us consider the generic optimization problem
\begin{equation} \label{eq:anyJ}
    \min_{\f\in\Xdual} J(\meas,\f),
\end{equation}
with a cost functional $J\colon\reals^M\times\Xdual\to \reals$.
We assume here that the optimization is performed over a
Banach space $\Xdual$ \emph{with a predual $\X$}, in
accordance with a more general version of Theorem~\ref{th:rep}
(see~\cite{Unser2021}). However, this does not impact the applicability of the
results in this section to \eqref{eq:Ellpproblem} because $\Ell_p(\Omega)$ spaces are
reflexive.
For the cost functional $J$, we assume that
\begin{enumerate}[({B}1)]
    \item for any given $\meas\in\reals^M$, there is a unique solution
        $\f_\meas$ to \eqref{eq:anyJ}. Thus, the solution map
        $\solutionmap\colon \reals^M \to \Xdual$ with $\meas \mapsto \f_\meas$
        is well defined. \label{ass:uniquesol}
    \item The cost function $J$ is differentiable through its second
        argument, and there is a constant $\LipConstant>0$ such that,
        if
        \begin{equation} \label{eq:subdiff_ij}
            \el_{i,j}= \nabla_\f \left\lbrace
                                        J(\meas_i,\f)
                                    \right\rbrace(\f_{\meas_j}) \in \Xdual'
           \mbox{ for }i,j\in\lbrace1,2\rbrace\,,
        \end{equation}
        we have that, for any $\meas_1, \meas_2 \in \reals^M$,
        \begin{equation}\label{eq:Lipschitzgradient}
            \Vert \el_{1,1} - \el_{2,1} \Vert_{\Xdual'} \leq
            \LipConstant \Vert \meas_1 - \meas_2 \Vert_2\,.
        \end{equation}
        \label{ass:lipgrad}
        \item There is a subset $Y\subset \reals^M$ and constants
        $\alpha\geq 2$ and $C(Y)\in \reals_+$ such that
        \begin{equation} \label{eq:sconvex-bis}
            \left\langle \el_{1,1} - \el_{1,2},
                    \f_{\meas_1} - \f_{\meas_2}
            \right\rangle_{\Xdual'\times\Xdual}
            \geq C(Y) \Vert \f_{\meas_1} - \f_{\meas_2} \Vert_{\Xdual}^\alpha
        \end{equation}
        for all $\meas_1,\meas_2 \in Y$.
        \label{ass:sconvex}
\end{enumerate}

\begin{theorem}[Local H\"{o}lder Stability of
                \eqref{eq:anyJ}]\label{lem:stable_abstract}
    Consider the variational problem \eqref{eq:anyJ}. Assume
    (B\ref{ass:uniquesol})--(B\ref{ass:sconvex}) are fulfilled.
    Then,
    \begin{equation}\label{eq:stability_wsconvex}
        \Vert \f_{\meas_1} - \f_{\meas_2} \Vert_{\Xdual} \leq
        \left(\frac{\LipConstant}{C(Y)}\right)^{\frac{1}{\alpha-1}}
        \Vert \meas_1 - \meas_2 \Vert_2^{\frac{1}{\alpha-1}}
    \end{equation}
    for all $\meas_1,\meas_2 \in Y$, and $\solutionmap$ is Hölder continuous
    on $Y$.
\end{theorem}
\begin{proof}
    From the optimality of $\f_{\meas_i}$, we have that
     $\el_{i,i}=0$ for $i \in\lbrace1,2\rbrace$.
    Then, we have that
    \begin{align}
        C(Y) \Vert \f_{\meas_1} - \f_{\meas_2} \Vert_{\Xdual}^\alpha
        &\leq \bigl\langle \el_{2,1} - \el_{2,2},
                            \f_{\meas_1} - \f_{\meas_2}
\bigr\rangle_{\Xdual'\times\Xdual} \label{eq:lsconvex}\\
        &\leq \Vert \el_{2,1} - \el_{2,2} \Vert_{\X''}
                \Vert \f_{\meas_1} - \f_{\meas_2} \Vert_{\Xdual} \label{eq:ldualitybound} \\
        &= \Vert \el_{2,1} - \el_{1,1} \Vert_{\X''}
            \Vert \f_{\meas_1} - \f_{\meas_2} \Vert_{\Xdual} \label{eq:lzeros}\\
        & \leq \LipConstant \Vert \meas_1 - \meas_2 \Vert_2
            \Vert \f_{\meas_1} - \f_{\meas_2} \Vert_{\Xdual} \label{eq:llipschitz},
    \end{align}
    where \eqref{eq:lsconvex} makes use of \eqref{eq:sconvex-bis},
    \eqref{eq:ldualitybound} follows from the duality bound in
    Definition~\ref{def:duality_bound}, \eqref{eq:lzeros}
    is obtained from $\el_{1,1} = \el_{2,2} = 0$, and
    \eqref{eq:llipschitz} makes use of \eqref{eq:Lipschitzgradient}.
    The statement \eqref{eq:stability_wsconvex} then follows readily by
    rearrangement of the terms above.
\end{proof}

\begin{remark}
    The $\alpha$-uniform convexity of $J$ with respect to $\f$ implies
    (B\ref{ass:sconvex}). (See~\cite{Zalinescu2002} for an
    extensive treatment of uniformly convex functions, and particularly
    Corollary~3.5.11 for a detailed account of equivalent conditions.)
\end{remark}

\subsection{Growth of the Gradient of
        \texorpdfstring{$\Ell_p$}{Lp} Regularizers} \label{app:MVT}
        % !TEX root = ../main.tex
In our proofs, we need to verify that \eqref{eq:Ellpproblem} fulfills
(B\ref{ass:sconvex}). To do so, we rely on two results on the
gradient of the regularizer in \eqref{eq:Ellpproblem}, namely, on the
gradient of the $p$th power of the $\Ell_p$ norm.
In order to simplify the notation, we introduce a useful function
in Definition~\ref{def:gp}.

\begin{definition}[Gradient of the $\Ell_p$ Regularizer]
\label{def:gp}
	Let $g_p\colon \reals \rightarrow \reals$ be such that
	\begin{equation}\label{eq:gp}
		g_p(x) = \sign(x) \vert x \vert^{p-1} \mbox{ for any }x\in\reals.
	\end{equation}
	Here, $\sign(x)=x/\vert x\vert$ for
	$x\in\reals\setminus\lbrace0\rbrace$ and
	$\sign(0)=0$.
\end{definition}
Then, the gradient of the $\Ell_p$ regularizer evaluates to
$$
    \nabla_\f\bigl\lbrace \Vert \f \Vert^p_{\Ell_p}\bigr\rbrace
    = p ( g_p \circ f ).
$$

\begin{remark}
    Since $g_p \circ \f$ is proportional to the differential of a functional
    defined on $\Ell_p(\Omega)$ at $\f\in\Ell_p(\Omega)$, we have that
    $g_p\circ\f\in\Ell_q(\Omega)$ for any $\f\in\Ell_p(\Omega)$ (see also
    Remark~\ref{rem:gradient}).
\end{remark}

Together, Proposition~\ref{prop:gp2toinf} and Lemma~\ref{lem:gp1to2}
characterize the function in \eqref{eq:gp} for $p\in(1,\infty)$. First, a
known result~\cite{Zal83} bounds the growth of $g_p$ for $p\in[2,\infty)$.

\begin{proposition}[Growth of $g_p$ for
                    $p\in[2,\infty)$~\cite{Zal83}] \label{prop:gp2toinf}
    The function $g_p$ in Definition~\ref{def:gp}, for $x,y\in\reals$
    and $p\in[2,\infty)$, satisfies that
    \begin{equation}\label{eq:lower_bound_geq}
        (g_p(x) - g_p(y))(x-y) \geq 2^{2-p} \vert x-y \vert^p.
    \end{equation}
\end{proposition}

Second, we show that $g_p$ also grows controllably for $p\in(1,2)$.

\begin{lemma}[Growth of $g_p$ for $p\in(1,2)$]\label{lem:gp1to2}
    The function $g_p$ in Definition~\ref{def:gp}, for $x,y\in\reals$
    and $p\in(1,2)$, satisfies that
    \begin{align}\label{eq:subdiff_prop}
        g_p(x) - g_p(y) = (p-1)\, z^{p-2}
                \left(x - y\right)
    \end{align}
    for some $z\leq \vert x \vert + \vert y \vert$.
\end{lemma}
\begin{proof}
    Note that $g_p$ is a differentiable function in
    $\reals\setminus\lbrace0\rbrace$, with derivative
    $g_p'(x) = (p-1) \vert x \vert^{p-2}$.

    If $xy\geq0$ (\emph{i.e.}, $x$ and $y$ have the same sign or one is zero), then
    \eqref{eq:subdiff_prop} corresponds to the statement of the mean-value
    theorem with $z\in[\vert x \vert, \vert y \vert]$ and, thus,
    $z\leq \vert x \vert + \vert y \vert$.

    If $xy<0$, assume without loss of generality that $x>0>y$. Then, by applying
    the mean-value theorem twice (once between $x$ and $0$, and once between
    $0$ and $y$), we obtain that
    \begin{align*}
        \Delta g_p = g_p(x) - g_p(y) = (p-1) \bigl(z_1^{p-2} x
                - z_2^{p-2} y \bigr),
    \end{align*}
    with $z_1\in[0,\vert x\vert]$ and $z_2\in[0,\vert y \vert]$. Then,
    let $\alpha = \frac{x}{x - y}$ and note that
    \begin{align*}
        \Delta g_p  =  (p-1) \bigl(\overbrace{z_1^{p-2} \alpha
            + z_2^{p-2}(1-\alpha)}^{\tilde{\alpha}}\bigr)
        \cdot\,\bigl(x - y \bigr).
    \end{align*}
    Because $\alpha\in[0,1]$, we know that
    $\tilde\alpha\in[z^{p-2}_1,z^{p-2}_2]$. Then, by the intermediate-value
    theorem applied to the continuous function $x\mapsto x^{p-2}$, we know that
    there is a point $z\in[z_1,z_2]$ such that $z^{p-2} = \tilde\alpha$.
    Furthermore, we have that
    $z\leq \max_{i\in\lbrace1,2\rbrace}z_i\leq \vert x \vert + \vert y \vert$.
\end{proof}

\subsection{General Statements and Proofs}

    We are now equipped to present the generalizations of
    Theorems~\ref{th:lp2toinf} and~\ref{th:lp1to2} to
    \eqref{eq:Ellpproblem}, together with their proofs.
    In both cases, the proof has the same structure. First,
    Theorem~\ref{th:rep} guarantees (B\ref{ass:uniquesol}) due to the
    properties of the $\Ell_p(\Omega)$ spaces and assumptions
    (A\ref{ass:niceE}) and (A\ref{ass:indepmeasop}).
    Second, (B\ref{ass:lipgrad}) follows because $E$ has a Lipschitz-continuous gradient
    (A\ref{ass:LipschitzdiffE}) and the regularizer does not depend on $\meas$.
    Finally, we show that (B\ref{ass:sconvex}) is fulfilled by exploiting the
    characterization of $g_p$ in Section~\ref{app:MVT}.
    Then, because (B\ref{ass:uniquesol})--(B\ref{ass:sconvex}) are
    fulfilled, we can apply Theorem~\ref{lem:stable_abstract}.
    Note that
    \begin{equation} \label{eq:derivativeLp}
       \nabla_\f \left\lbrace J(\meas,\f)\right\rbrace =
                                           \nabla_\f \left\lbrace
                                                E(\meas,\measop(\f))
                                                \right\rbrace
            + \lambda\, p\, g_p(\f)\,,
    \end{equation}
    where $g_p$ is applied pointwise.

    \begin{theorem}[Hölder Continuity of $\Ell_p$-Regularized Linear Inverse
        Problems, $p\in[2,\infty)$] \label{th:Lp2toinf}
        Consider the variational problem \eqref{eq:Ellpproblem} with
        $p\in[2,\infty)$ and assume that
        (A\ref{ass:niceE})--(A\ref{ass:indepmeasop}) are fulfilled. Then,
        for any given measurement $\meas\in\reals^M$, \eqref{eq:Ellpproblem}
        has a unique solution $\f_\meas\in\Ell_p(\Omega)$. Moreover, it holds that
        $$
            \Vert \f_{\meas_1} - \f_{\meas_2} \Vert_{\Ell_p} \leq
            \left(
                \frac{2^{p-2}\LipConstant_p}{\lambda p}
            \right)^{\frac{1}{p-1}}
            \Vert \meas_1 - \meas_2 \Vert_2^{\frac{1}{p-1}},
        $$
        for any two $\meas_1,\meas_2\in\reals^M$.
    \end{theorem}
    \begin{proof}
        Because $\X=\Ell_p(\Omega)$ is a strictly convex
        space and because $\psi\colon x\mapsto \lambda x^{p}$ with
        $\lambda\in\nnreals$ is strictly convex, Theorem~\ref{th:rep}
        tells us that (A\ref{ass:niceE}) and (A\ref{ass:indepmeasop})
        imply that \eqref{eq:Ellpproblem} has a unique solution $\f_\meas$
        and, thus, that (B\ref{ass:uniquesol}) is fulfilled.
        Assumption (A\ref{ass:LipschitzdiffE}) implies that
        (B\ref{ass:lipgrad}) is fulfilled with constant $\LipConstant_p$.

        Now, we show that (B\ref{ass:sconvex}) is fulfilled using
        Proposition~\ref{prop:gp2toinf}. Because $E$ is convex in its
        second argument, we have that
        $\langle
    \nabla_f\left\lbrace E(\meas_1, \measop(\f))\right\rbrace(\f_{\meas_1})
    -\nabla_f\left\lbrace E(\meas_1, \measop(\f))\right\rbrace(\f_{\meas_2}),
          \f_{\meas_1} - \f_{\meas_2} \rangle_{\Ell_q\times\Ell_p} \geq 0 $,
        and using the decomposition \eqref{eq:derivativeLp} for both
        $\el_{1,1}$ and $\el_{1,2}$, we obtain that
        \begin{align}
            \langle
              \el_{1,1} - \el_{1,2},  \f_{\meas_1} &- \f_{\meas_2}
            \rangle_{\Ell_q\times\Ell_p}
            \geq \nonumber \\
            & \lambda p \langle
                            g_p(\f_{\meas_1}) - g_p(\f_{\meas_2}),
                            \f_{\meas_1} - \f_{\meas_2}
                        \rangle_{\Ell_q\times\Ell_p}
            \label{eq:onlygmatters}\,,
        \end{align}
        where $\Ell_q(\Omega)$ is the Lebesgue function space that identifies with the dual of
        $\Ell_p(\Omega)$ (see Example~\ref{ex:lp_dual_lq}), with $q=1/(1-1/p)$.
        Thus, applying Proposition~\ref{prop:gp2toinf} pointwise, we obtain that
        \begin{equation*}
            \left\langle
                \el_{1,1} - \el_{1,2}, \f_{\meas_1} - \f_{\meas_2}
            \right\rangle_{\Ell_q\times\Ell_p}
             \geq \lambda p \, 2^{2-p}
                \Vert \f_{\meas_1} - \f_{\meas_2} \Vert_{\Ell_p}^p
        \end{equation*}
        and (B\ref{ass:sconvex}) is satisfied with $\alpha = p$ and
        $C(Y) = C = \lambda p\, 2^{2-p}$. Thus,
        Theorem~\ref{lem:stable_abstract} applies and the proof is complete.
    \end{proof}

    \begin{theorem}[Local Lipschitz Continuity of $\Ell_p$-Regularized Linear
        Inverse Problems, $p\in(1,2)$] \label{th:Lp1to2}
        Consider the variational problem \eqref{eq:Ellpproblem} with
        $p\in(1,2)$ and assume that
        (A\ref{ass:niceE})--(A\ref{ass:contE}) are fulfilled. Then,
        for any given measurement $\meas\in\reals^M$, \eqref{eq:Ellpproblem}
        has a unique solution $\f_\meas\in\Ell_p(\Omega)$. Moreover, for any closed
        and bounded set $Y\subset\reals^M$ of measurements, there is a constant
        $r_p(Y)\in\nnreals$ such that
        $$
            \Vert \f_{\meas_1} - \f_{\meas_2} \Vert_{\Ell_p} \leq
            \frac{(2\,r_p(Y))^{2-p}\LipConstant_p}{\lambda p(p-1)}
            \Vert \meas_1 - \meas_2 \Vert_2,
        $$
        for any two $\meas_1,\meas_2\in Y$.
    \end{theorem}
    \begin{proof}
        By the same arguments as in the proof of Theorem~\ref{th:Lp2toinf},
        (B\ref{ass:uniquesol}) is fulfilled, \eqref{eq:Ellpproblem} has a
        unique solution $\f_\meas$, and (B\ref{ass:lipgrad}) is fulfilled
        with constant $\LipConstant_p$.

        Now, we use (A\ref{ass:contE}) and Lemma~\ref{lem:gp1to2} to show that
        (B\ref{ass:sconvex}) is fulfilled. Consider
        $\tilde\f\in\Ell_p(\Omega)$ and recall that (A\ref{ass:contE}) specifies that
        $E(\cdot,\measop(\tilde\f))$ is continuous. Then, it holds that
        \begin{align}
            \lambda \Vert \f_\meas \Vert_{\Ell_p(\Omega)}^p & \leq
                E(\meas,\measop(\f_\meas))
                + \lambda \Vert \f_\meas \Vert_{\Ell_p}^p\nonumber\\
            & \leq  \max_{\meas \in Y} J(\meas,\tilde{\f}), \label{eq:contbound}
        \end{align}
        for all $\meas \in Y$. The last term in \eqref{eq:contbound} is finite
        because $Y$ is compact and the Weierstrass extreme-value theorem applies.
        Therefore, there is $r_p(Y)\in\nnreals$ such that, for every
        $\meas\in Y$,
        \begin{equation}\label{eq:bounded_regularizer}
            \Vert \f_\meas \Vert_{\Ell_p(\Omega)} \leq
            r_p(Y) \leq \Bigl(
                        \max_{\meas \in Y}
                                J(\meas,\tilde \f)/\lambda\Bigr)^{\frac1p}.
        \end{equation}
        We now combine \eqref{eq:bounded_regularizer} with Lemma~\ref{lem:gp1to2} to show
        that (B\ref{ass:sconvex}) is fulfilled.  For any two given
        $\meas_1,\meas_2\in Y$, let
        \begin{itemize}
        		\item $\xi = \vert \f_{\meas_1}  \vert + \vert \f_{\meas_2} \vert \in \Ell_p(\Omega)$;
		\item $q=1/(1-1/p)$, the index of the $\Ell_q(\Omega)$ space that identifies with the dual of
			$\Ell_p(\Omega)$;
		\item $\beta = p(2-p)/2$;
		\item $t=2/(2-p)$ and $s= 2/p = 1/(1-1/t)$.
	\end{itemize}
	Then,
        \begin{align}
            \!\!\!\Vert &\f_{\meas_1} - \f_{\meas_2} \Vert_{\Ell_p}^2
            =  \biggl(\int_\Omega \bigl\vert \f_{\meas_1} -
                                                \f_{\meas_2}
                                    \bigr \vert^p \dx \mu \biggr)^\frac{2}{p}
            \notag\\
            &\leq \biggl(\left\Vert \xi^\beta\right\Vert_{\Ell_t}
                        \biggl\Vert
                            \frac{\vert\f_{\meas_1}-\f_{\meas_2}\vert^p}{
                                     \xi^\beta}
                        \biggr\Vert_{\Ell_s}
                    \biggr)^\frac{2}{p} \label{eq:Holder} \\
            & = \Vert \xi \Vert^{2-p}_{\Ell_p}
                \int_\Omega \xi^{p-2}
                            \bigl(\f_{\meas_1}-\f_{\meas_2}\bigr)^2 \dx \mu
                \notag\\
            & \leq \frac{(2\,r_p(Y))^{2-p}}{p-1}
                \bigl\langle g(\f_{\meas_1}) - g(\f_{\meas_2}),
                                \f_{\meas_1} - \f_{\meas_2}
                \bigr\rangle_{\Ell_q\times\Ell_p}\label{eq:gp_application}\\
            &\leq \frac{(2r_p(Y))^{2-p}}{\lambda p(p-1)}
                \bigl\langle \el_{1,1} - \el_{1,2},
                            \f_{\meas_1} - \f_{\meas_2}
                \bigr\rangle_{\Ell_q\times\Ell_p}, \label{eq:est_growth}
        \end{align}
        for any $\el_{1,1}$, $\el_{1,2}$. Here, $\mu$ represents the Lebesgue
        measure on $\reals^d$. In \eqref{eq:Holder}, we
        use that $1/t+1/s=1$ and apply the Hölder inequality for
        $\Ell_t(\Omega)$ and $\Ell_s(\Omega)$. Then, in
        \eqref{eq:gp_application}, we use Lemma~\ref{lem:gp1to2} pointwise
        and apply \eqref{eq:bounded_regularizer}. Then, in \eqref{eq:est_growth}
        we conclude, using \eqref{eq:onlygmatters} and the fact that $E$ is convex in
        its second argument. Thus, (B\ref{ass:sconvex}) is satisfied with
        $\alpha = 2$ and $C(Y) = \lambda p (p-1)(2\,r_p(Y))^{p-2}$. Thus,
        Theorem~\ref{lem:stable_abstract} applies and the proof is complete.
    \end{proof}

\section{Conclusions} \label{sec:conclusion}
    We have shown that $\ell_p$-regularized
strategies with $p\in(1,\infty)$ in linear inverse problems have good
stability properties. The strongest guarantees are those given by Tikhonov
regularization in Hilbert spaces ($\ell_2$ regularization), for which the
reconstruction map is globally Lipschitz continuous. For $p\in(1,2)$, the
reconstruction map is still Lipschitz continuous, albeit only locally in the
space of measurements $\reals^M$. For $p\in(2,\infty)$, the reconstruction map
is globally $1/(p-1)$-H\"{o}lder continuous. Thus the stability claim is
stronger for $p$ closer to $2$. To the best of our knowledge, our bounds are
currently the strongest stability results for $\ell_p$ regularization for
$p\in(1,\infty)$. That said, we have not yet investigated the tightness
of our bounds for the different regimes, and the option of improved
bounds for $p\neq 2$ remains open.

\section{Acknowledgments}

    The authors would like to thank Dr.~Rahul~Parhi for
    	fruitful discussions on representer theorems.

    We acknowledge access to the facilities and expertise of the CIBM Center
for Biomedical Imaging, a Swiss research center of excellence founded and
supported by Lausanne University Hospital (CHUV), University of Lausanne
(UNIL), École polytechnique fédérale de Lausanne (EPFL), University of
Geneva (UNIGE), and Geneva University Hospitals (HUG).

    The research leading to these results was partly supported by the
    European Research Council (ERC) under European Union’s Horizon 2020
    (H2020), Grant Agreement - Project No 101020573 FunLearn.

\appendices
    \section{Mathematical Preliminaries}
        \label{app:premath}
        % !TEX root = ../main.tex
Here, we give a digest of the technical definitions and results that
are most relevant to our work.

\begin{definition}[Banach Space] \label{def:banach}
    A Banach space $\X$ is a vector space with norm $\Vert\cdot\Vert_\X$
    for which it is complete.
\end{definition}

A Banach space is also a complete metric space, as its norm defines a distance
$d\colon\X\times\X\to\reals_+$ via
$d(\f_1,\f_2)=\left\Vert \f_1 - \f_2 \right\Vert_\X$ for any two
$\f_1,\f_2\in\X$. In turn, this distance defines an associated
topology---open sets and convergence are defined in terms of the
distance $d$. The completeness must be understood with respect to the
distance $d$, so that Cauchy sequences in the sense of $d$ must converge in
$\X$. Banach spaces are instrumental to our work.

\begin{definition}[Hilbert Space]\label{def:hilbert}
    A Hilbert space $\H$ is a Banach space with an inner product
    $\langle\cdot,\cdot\rangle_{\H}\colon\H\times\H\to\reals$. Its norm
    is induced by the inner product as
    $\Vert\varphi\Vert_\H \coloneqq \sqrt{\langle\varphi,\varphi\rangle_\H}$ for any $\varphi\in\H$.
\end{definition}

Hilbert spaces are simpler special cases of Banach spaces, in the
sense that many of the properties from finite-dimensional Euclidean spaces
apply. The inner product $\langle\cdot,\cdot\rangle_{\H}$ allows us to quantify
the alignment between two vectors. Specifically, one can define an angle $\theta$
between two vectors $\varphi_1,\varphi_2\in\H$ by
\begin{equation}\label{eq:angle}
    \cos(\theta) = \frac{\langle \varphi_1,\varphi_2\rangle_\H}{
                         \Vert \varphi_1 \Vert_\H \Vert \varphi_2 \Vert_\H},
\end{equation}
even when these vectors in $\H$ may be infinite-dimensional.

\begin{definition}[Continuous Dual of a Banach Space]\label{def:dual}
    The dual space $\Xdual$ of a Banach space $\X$ is the vector space formed
    by all linear and continuous functionals $\nu\colon\X\to\reals$.
    The corresponding values $\nu(\f)$ are often
    expressed in terms of the bilinear form (the duality product)
    $\langle\cdot,\cdot\rangle_{\Xdual\times\X}\colon\Xdual\times\X\to\reals$ given
    by $\langle \nu, \f \rangle_{\Xdual\times\X} \coloneqq \nu(\f)$.
    The dual space $\Xdual$ is, in turn, a Banach space, when equipped with the operator norm
    \begin{equation}\label{eq:dual_norm}
        \Vert \nu \Vert_{\Xdual} \coloneqq
        \sup_{\f\in\X\setminus\lbrace 0 \rbrace}
            \frac{\vert \langle \nu, \f \rangle_{\Xdual\times\X} \vert}{
                  \Vert \f \Vert_{\X}}.
    \end{equation}
\end{definition}
\begin{example}[$\Ell_q(\Omega)$ is the Dual of $\Ell_p(\Omega)$ with
    $q=p/(p-1)$, $p,q\in(1,\infty)$] \label{ex:lp_dual_lq}
    The concept of a dual vector space can be made more concrete by
    identifying it with another vector space. Most relevant to our work is
    $\X=\Ell_p(\Omega)$. Then, the Riesz theorem asserts that every $\nu\in\Xdual$
    can be expressed, $\forall \f \in \X$, as
    $$
        \langle \nu, \f \rangle_{\Xdual\times\X} =
        \int_\Omega \tilde\nu \,\f\, \dx \mu
    $$
    for some $\tilde\nu\in\Ell_q(\Omega)$, and
    $\Vert\nu\Vert_{\Xdual} = \Vert \tilde\nu\Vert_{\Ell_q(\Omega)}$.
    Because of this one-to-one isomorphism, we abuse the notation and
    write that $\Xdual=\Ell_q(\Omega)$ and $\tilde\nu=\nu$. This is common practice
    but it is to handle with care, as these are different mathematical
    objects.
\end{example}

\begin{proposition}[Duality Bound] \label{def:duality_bound}
    For any $\nu\in\Xdual$ and $\f\in\X$, one has that
    $$
        \vert \langle \nu,\f\rangle_{\Xdual\times\X}\vert \leq
        \Vert \nu \Vert_{\Xdual} \Vert \f \Vert_\X,
    $$
    which is called the duality bound.
    This bound is sharp for any dual pair of Banach spaces $(\X,\Xdual)$.
\end{proposition}
\begin{proof}
    The duality bound follows immediately from \eqref{eq:dual_norm}. Sharpness is a corollary of the Hahn-Banach theorem.
\end{proof}

The duality bound generalizes the Cauchy-Schwartz inequality in Hilbert spaces
$$
    \vert \langle \varphi_1,\varphi_2\rangle_\H \vert\leq
        \Vert \varphi_1 \Vert_\H \Vert \varphi_2 \Vert_\H,
    \forall \varphi_1,\varphi_2\in\H.
$$
The Cauchy-Schwartz inequality is saturated by parallel vectors,
corresponding to $\lvert\cos(\theta)\rvert=1$ in \eqref{eq:angle}.
Analogously, it is worthwhile to identify the set of dual vectors that saturate the duality
bound for a given vector $\nu\in\X$.

\begin{definition}[Duality Map] \label{def:duality_map}
    The duality map is the set-valued map
    $\dualitymap\colon\X\rightrightarrows\Xdual$ given by
    \begin{equation*}
        \dualitymap(\f) =\left\lbrace\nu\in\Xdual:
            \begin{array}{l}
            \Vert \nu \Vert_{\Xdual} = \Vert \f \Vert_\X\mbox{ and }\\
            \langle \nu,\f \rangle_{\Xdual\times\X} = \Vert \nu \Vert_{\Xdual}
                                                     \Vert \f\Vert_\X
            \end{array}
            \right\rbrace.
    \end{equation*}
\end{definition}

As we have seen in Theorem~\ref{th:rep}, the duality map characterizes the set
of solutions of the variational problems in Banach spaces that take the form \eqref{eq:solutionset}.
For strictly convex Banach spaces $\X$, the duality map is single-valued.

\begin{proposition}[Duality Map in $\Ell_q(\Omega)$,
$q\in(1,\infty)$ {\cite[Corollary~4.10]{Cioranescu1990}}]
\label{prop:duality_map_ellq}
    The duality map for $\Xdual=\Ell_q(\Omega)$ with $q\in(1,\infty)$ is
    single-valued and given by
    \begin{equation}\label{eq:lpdualitymap}
        \dualitymapdual(\nu) =
            \frac{|\nu|^{q-1}}{
                  \Vert \nu \Vert_{\Ell_q}^{q-2}}
            \sign(\nu),
    \end{equation}
    where the absolute value $\vert\cdot\vert$ and $\sign(\cdot)$ operators
    are applied element-wise.
\end{proposition}

The finite-dimensional discrete equivalent of \eqref{eq:lpdualitymap} was
implemented to obtain the visualizations in
Figures~\ref{fig:simple_duality_map} and \ref{fig:manifold_vs_a}. There, we
see that the effect of the duality map for $\Xdual=\ell_q$ resembles the behavior that
we expect from $\ell_p$-norm regularization in variational problems for $1/p+1/q=1$.

\begin{definition}[Riesz Map and Hilbert Spaces]\label{def:riesz_map}
    For a Hilbert space $\X=\H$, the duality map $\dualitymap$ is
    single-valued. Furthermore, its inverse
    $\Rieszmap=\dualitymap^{-1}\colon\H'\to\H$ is called the Riesz map. It maps a
    dual vector $\nu\in\H'$ to its Riesz representer $\tilde\nu\in\H$, such that
    $$
        \langle \nu,\f \rangle_{\H'\times\H} = \langle \tilde\nu,\f\rangle_\H
    $$
    for all $\f\in\H$.
\end{definition}

This is the equivalent of the phenomenon
described in Example~\ref{ex:lp_dual_lq}, but in Hilbert spaces.

\begin{definition}[Reflexive Banach Space] \label{def:reflexive}
    A Banach space $\X$ is reflexive if it can be identified with its bidual $\Xdual'$.
\end{definition}

The bidual space $\Xdual'$ is simply the continuous dual of $\Xdual$:
that is, it is formed of all the continuous linear functionals
$\tilde{f}\colon\Xdual\rightarrow\reals$. In particular, 
we can construct such a functional from any $\f\in\X$
through the identity
$$
    \langle \tilde\f, \nu \rangle_{\Xdual'\times\Xdual} =
    \langle \nu,\f \rangle_{\Xdual\times\X}\,,
$$
which specifies the canonical embedding of $\X$ into $\Xdual'$.
Although not formally accurate in the same sense as
Example~\ref{ex:lp_dual_lq}, this is usually denoted as $\X\subseteq\Xdual'$.
In reflexive Banach spaces, it is in the same sense that $\X=\Xdual'$: all linear and
continuous functionals on $\Xdual$ can be represented by an element of
$\X$
through the canonical embedding, which is then one-to-one.

For a general Banach space $\X$, a more general version of Theorem~\ref{th:rep}
(\emph{cf.}~\cite{Unser2021}) restricts the
choice of the measurement functionals (of $\f\in\Xdual$) by assuming that
$\lbrace\nu_m\rbrace_{m=1}^M \subset\X\subseteq\Xdual'$. However, this is not
a limitation in reflexive Banach spaces because $\X=\Xdual'$, which yields the
formulation of Theorem~\ref{th:rep} in this paper.
In that more general version of Theorem~\ref{th:rep},
$E$ and $\psi$ are also not required to be strictly convex.
Then, the solution is no longer unique. Instead, the
solution set is guaranteed to be nonempty, convex, and weak*-compact. We
include here the definition of weak*-compact for completeness.

\begin{definition}[Weak* Compactness]\label{def:weakstar-compact}
    A weak*-compact set in the dual of a separable Banach space $\X$
    is a set $\mathcal{C}\subset\Xdual$ such that, for any sequence
    $\lbrace \nu_n \rbrace_{n=1}^\infty\subset\mathcal{C}$, there is a
    weak*-convergent subsequence $\lbrace \nu_{n_r}\rbrace_{r=1}^\infty$,
    meaning that
    $$
        \langle \nu_{n_r}, \f \rangle_{\Xdual\times\X} \rightarrow
        \langle \nu,\f \rangle_{\Xdual\times\X}\mbox{ as }r\rightarrow \infty
    $$
    for some $\nu\in\mathcal{C}$ and any $\f\in\X$.
\end{definition}

Weak* compactness is useful in variational theory because it guarantees the
existence of minimizers of certain cost functionals using the generalized
Weierstrass extreme-value theorem. This makes the result in~\cite{Unser2021}
more attractive, as the further selection among the solutions
of an initial variational problem is made possible through variational techniques.

\begin{remark}[Gradient of a
Functional]\label{rem:gradient}
    The gradient of a differentiable functional (in the Fr\'echet sense)
    $J\colon\Xdual\to\reals$ is the map $\nabla J\colon\Xdual\to\Xdual'$
    such that
    $$
        \lim_{\Vert \tilde \f \Vert_{\Xdual}\rightarrow 0} \left\lbrace
        \frac{J(\f+\tilde\f) - J(\f) -
            \langle \nabla J(f), \tilde\f\rangle_{\Xdual'\times\Xdual}
             }{\Vert\tilde\f\Vert_{\Xdual}}\right\rbrace = 0.
    $$
\end{remark}

This contextualizes the proof of Theorem~\ref{lem:stable_abstract}, where we
explicitly treat gradient and subdifferential values as elements of
$\Xdual'$. In the main body of the paper, we can avoid this explicit
treatment because of the reflexivity of the spaces being considered.

%\added{
%Under the setup of \eqref{eq:solutionset}, it is not always obvious how to choose the
%pair of Banach spaces $(\X,\Xdual)$ so that they yield the desired regularization but
%still fulfill the condition $\nu_m\in\Xdual$. }

%\begin{proposition}[Regularity Conditions for Radon Measurements] \label{prop:Radon}
%	\added{Let $\lbrace\boldsymbol{\theta}_m\rbrace_{m=1}^M\subset S$ and
%	$\lbrace t_m\rbrace_{m=1}^M\subset [-1,1]$, where $S$ is the unit circle
%    in $\reals^2$, and consider the Radon measurement vectors
%    $\nu_m = \delta(t_m-\langle \cdot, \boldsymbol{\theta}_m\rangle)$
%	for $m\in\lbrace1,2,\dots,M\rbrace$. Then, if we choose the Sobolev space
%	$\X=H_0^{s}(\Omega)$ with $s>1/2$ on the closed unit ball $\Omega=\lbrace
%    x\in\reals^2 : \Vert x\Vert_2\leq1\rbrace$,
%	we have that $\nu_m\in\Xdual$ for any $m$.}
%\end{proposition}
%\begin{proof}
%	\added{A sufficient condition for $\nu_m\in\Xdual$ is that the space of
%    Radon transforms $\mathcal{R}(\X)$ embeds continuously into
%    $C^0(S\times [-1,1])$.
%    This embedding ensures that sampling at $(\boldsymbol{\theta}_m,t_m)$
%    is properly defined. Theorem~5.1 of \cite{Natterer1986} states that
%    $\mathcal{R}(H^s_0(\Omega))$ embeds
%    continuously into $H^{s+1/2}(S\times[-1,1])$. Because of the
%    Sobolev embedding theorem, we know that
%    $H^{s+1/2}(S\times[-1,1])\subset C^0(S\times [-1,1])$ for any $s>1/2$.}
%\end{proof}
%
%\added{Proposition~\ref{prop:Radon} generalizes well to any compact
%$\Omega\subset\reals^2$.}

    \bibliographystyle{IEEEtran}
    \bibliography{stability}
\end{document}